\documentclass[11pt]{article}

\usepackage{amsmath,amssymb,amsthm,mathrsfs}
\usepackage[margin=1in]{geometry}
\usepackage{hyperref}

\hypersetup{
  colorlinks=true,
  linkcolor=blue,
  citecolor=blue,
  urlcolor=blue
}
\newtheorem{theorem}{Theorem}[section]
\newtheorem{proposition}[theorem]{Proposition}
\newtheorem{lemma}[theorem]{Lemma}
\newtheorem{corollary}[theorem]{Corollary}
\theoremstyle{definition}
\newtheorem{definition}[theorem]{Definition}
\theoremstyle{remark}
\newtheorem{remark}[theorem]{Remark}
\theoremstyle{definition}
\newtheorem{example}[theorem]{Example}
\theoremstyle{plain}

\newtheorem*{maintheorem}{Theorem A}

\newcommand{\Hom}{\mathrm{Hom}}

\newcommand{\Z}{\mathrm{Z}}
\newcommand{\C}{\mathrm{C}}

\newcommand{\Prob}{\mathbf{P}}

\newcommand{\Comm}{\mathrm{Comm}}

\newcommand{\Alt}{\mathrm{Alt}}
\newcommand{\citedin}[1]{\newblock \emph{Cited in:} #1.}

\makeatletter
\def\@fnsymbol#1{%
  \ifcase#1\or
    \TextOrMath\textasteriskcentered \ast%
  \or
    \TextOrMath{\textasteriskcentered\textasteriskcentered}{\ast\ast}%
  \else
    \@ctrerr
  \fi
}

\makeatother

\newcommand{\Author}[2]{%
    {#1}\thanks{Email: \href{mailto:#2}{\texttt{#2}}}\\\Affiliation}

\newcommand{\Department}[1]{\def\DeptName{#1}}
\newcommand{\Faculty}[1]{\def\FacName{#1}}
\newcommand{\University}[1]{\def\UnivName{#1}}

\Department{Department of Mathematics}
\Faculty{Faculty of Natural Sciences}
\University{Ariel University, Ariel, Israel}

\author{
\Author{Vadim E. Levit}
{levitv@ariel.ac.il}
\and 
\Author{Robert Shwartz}
{robertsh@ariel.ac.il}
}

\title{Higher Commutativity in Finite Groups:\\[0.2em]
Exact Asymptotics and Finite Spectrum}

\date{}

\begin{document}
\maketitle

\begin{abstract}
For a finite group $G$ and an integer $r\ge 1$ let
\[
P_r(G):=\frac{|\Hom(\mathbb Z^r,G)|}{|G|^r},
\qquad
\kappa_r(G):=\big|\Hom(\mathbb Z^r,G)/G\big|,
\]
where $G$ acts by diagonal conjugation. Write
\[
m(G):=\max\{|A|:A\le G\text{ abelian}\},\qquad
N_{\max}(G):=\#\{A\le G:\ A\text{ abelian},\ |A|=m(G)\},
\]
and let $b(G)<m(G)$ be the second-largest order of an abelian subgroup when $G$ is non-abelian. We prove the exact dominant-abelian asymptotic
\[
|\Hom(\mathbb Z^r,G)|=N_{\max}(G)\,m(G)^r+O\!\bigl(b(G)^r\bigr),
\]
hence
\[
\lim_{r\to\infty}P_r(G)^{1/r}=\frac{m(G)}{|G|},
\qquad
\lim_{r\to\infty}\kappa_r(G)^{1/r}=m(G).
\]
For non-abelian $G$ the rank-generating series
\[
\mathcal P_G(z):=\sum_{r\ge 2}P_r(G)z^{r-2}
\]
is rational, admits a finite partial-fraction expansion supported on abelian subgroup indices, and has first pole $z=|G|/m(G)$, with coefficient $N_{\max}(G)m(G)/|G|$ of the term $(|G|/m(G)-z)^{-1}$. We also prove that for every split abelian extension $G=A\rtimes K$,
\[
m(G)=\max_{B\le K\ \mathrm{abelian}} |\C_A(B)|\,|B|,
\]
so the dominant exponential base is already encoded by fixed-subgroup geometry in the acting quotient. When the acting quotient $K$ is abelian, we obtain the exact subgroup-lattice formula
\[
|\Hom(\mathbb Z^r,G)|=\sum_{B\le K}\lambda(B)\,|\C_A(B)|^r\,\varphi_r(B),
\]
where $\varphi_r(B)$ counts generating $r$-tuples of $B$ and $\lambda(B)$ counts abelian lifts of $B$ with full kernel $\C_A(B)$. Equivalently, the finite Dirichlet spectrum is explicit up to the coefficients $\lambda(B)$, and in the cyclic and coprime cases those coefficients admit closed formulas; in particular,
\[
N_{\max}(G)=\sum_{\substack{B\le K\\ |\C_A(B)|\,|B|=m(G)}}\lambda(B).
\]
In the coprime-action regime $(|A|,|K|)=1$, these lift multiplicities collapse to
\[
\lambda(B)=|A:\C_A(B)|\qquad (B\le K),
\]
so the exact count is determined entirely by the fixed-subgroup lattice. For cyclic quotients this specializes to the Jordan-totient formula
\[
|\Hom(\mathbb Z^r,A\rtimes C_\omega)|=\sum_{n\mid \omega}|A:\C_A(B_n)|\,|\C_A(B_n)|^r J_r(n),
\]
so in the one-generator case the lift multiplicity collapses to the fixed-stratum index. We also prove an inverse finite-spectrum theorem: the tail of the hierarchy $\{P_r(G)\}_{r\ge 2}$ determines the full abelian-index spectrum $\{(m,c_m(G))\}$. The same finite spectrum yields a linear recurrence of exact order equal to the number of nonzero coefficients, equivalently a finite-rank Hankel matrix for the hierarchy. Thus the higher-commutativity hierarchy has an exact leading asymptotic, a finite spectral decomposition, and a finite-determination principle controlled by the maximum-order abelian subgroups of $G$.
\end{abstract}



\section{Introduction}

For a finite group $G$ and an integer $r\ge 1$ define
\[
P_r(G):=\frac{|\Hom(\mathbb Z^r,G)|}{|G|^r},
\qquad
\kappa_r(G):=\big|\Hom(\mathbb Z^r,G)/G\big|,
\]
where $G$ acts by diagonal conjugation. Thus $P_1(G)=1$, $\kappa_1(G)=k(G)$, and $P_2(G)=k(G)/|G|$, the classical commuting probability of Erd\H{o}s--Tur\'an and Gustafson \cite{ErdosTuran1968,Gustafson1973}; more generally, $P_r(G)=d_{r-1}(G)$ in Lescot's notation \cite{Lescot1995} for every $r\ge 2$. Burnside's lemma gives
\[
P_{r+1}(G)=\frac{\kappa_r(G)}{|G|^r},
\]
so the probabilistic hierarchy $\{P_r(G)\}$ and the orbit-count hierarchy $\{\kappa_r(G)\}$ are the same data in two normalizations.
The classical invariant $P_2(G)$ has a substantial modern literature, including structural bounds for finite groups \cite{GuralnickRobinson2006} and global restrictions on the set of possible commuting probabilities \cite{Eberhard2015}. For semidirect products built from finite abelian groups, Nath computed the ordinary commutativity degree in a natural family answering a question of Lescot \cite{Nath2013}; the exact split-extension formulas below may be viewed as higher-rank analogues in which the whole hierarchy is controlled by fixed-subgroup data.

Another related direction studies generalized commuting probabilities attached to permutation equalities rather than to commuting tuples. Cherniavsky--Goldstein--Levit--Shwartz related these probabilities to Hultman numbers, and Shwartz--Levit developed the corresponding signed version \cite{CherniavskyGoldsteinLevitShwartz2017,ShwartzLevit2022}. The present hierarchy is different in source and structure: it counts homomorphisms from free abelian groups, so its spectral data are controlled by abelian subgroups of $G$.

This paper isolates the asymptotic and finite-spectrum side of higher commutativity. We vary the rank of the free abelian source $\mathbb Z^r$ and ask how the towers $\Hom(\mathbb Z^r,G)$ and $\kappa_r(G)$ are controlled by the abelian-subgroup geometry of $G$.

Write
\[
m(G):=\max\{|A|:A\le G\text{ abelian}\},\qquad
N_{\max}(G):=\#\{A\le G:\ A\text{ abelian},\ |A|=m(G)\},
\]
and let $b(G)<m(G)$ denote the second-largest order of an abelian subgroup when $G$ is non-abelian, with $b(G)=0$ when $G$ is abelian. Our main theorem identifies the exact dominant contribution.

\begin{maintheorem}[Asymptotic-spectral bridge]
For every finite group $G$,
\[
|\Hom(\mathbb Z^r,G)|=N_{\max}(G)\,m(G)^r+O\!\bigl(b(G)^r\bigr)
\qquad (r\to\infty).
\]
Equivalently,
\[
P_r(G)=N_{\max}(G)\Bigl(\frac{m(G)}{|G|}\Bigr)^r+O\!\Bigl(\Bigl(\frac{b(G)}{|G|}\Bigr)^r\Bigr),
\qquad
\kappa_r(G)=\frac{N_{\max}(G)m(G)}{|G|}\,m(G)^r+O\!\bigl(b(G)^r\bigr).
\]
Hence
\[
\lim_{r\to\infty}P_r(G)^{1/r}=\frac{m(G)}{|G|},
\qquad
\lim_{r\to\infty}\kappa_r(G)^{1/r}=m(G).
\]
\end{maintheorem}

The generating-series formulation is equally rigid: for non-abelian $G$ we prove a finite abelian-index spectrum
\[
P_r(G)=\sum_{m\ge 2}\frac{c_m(G)}{m^r},
\qquad
\mathcal P_G(z)=\sum_{m\ge 2}\frac{c_m(G)}{m(m-z)},
\]
so $\mathcal P_G(z)=\sum_{r\ge2}P_r(G)z^{r-2}$ is rational, its first pole is $z=|G|/m(G)$, and the coefficient of $(|G|/m(G)-z)^{-1}$ is $N_{\max}(G)m(G)/|G|$. Equivalently, $P_r(G)$ is a finite linear combination of exponentials $m^{-r}$ indexed by abelian subgroup indices; in particular, the hierarchy satisfies a finite linear recurrence and has finite Hankel rank. We further show that the tail of the sequence $\{P_r(G)\}_{r\ge 2}$ determines the entire finite abelian-index spectrum.

For split abelian extensions $G=A\rtimes K$ with $A$ abelian, we also prove
\[
m(G)=\max_{B\le K\ \mathrm{abelian}} |\C_A(B)|\,|B|,
\]
so the dominant exponential base is already visible in fixed-subgroup geometry in the acting quotient. When $K$ is itself abelian, we go further and obtain the exact subgroup-lattice formula
\[
|\Hom(\mathbb Z^r,G)|=\sum_{B\le K}\lambda(B)\,|\C_A(B)|^r\,\varphi_r(B),
\]
where $\varphi_r(B)$ counts generating $r$-tuples of $B$ and $\lambda(B)$ counts abelian lifts of $B$ with full kernel $\C_A(B)$. Equivalently,
\[
|\Hom(\mathbb Z^r,G)|=\sum_{B\le K}\sum_{D\le B}\lambda(B)\mu_B(D)\,(|\C_A(B)|\,|D|)^r,
\]
so the full finite Dirichlet spectrum is explicit up to the coefficients $\lambda(B)$ at the subgroup-lattice level; in the cyclic and coprime cases these coefficients have closed forms, and in particular
\[
N_{\max}(G)=\sum_{\substack{B\le K\\ |\C_A(B)|\,|B|=m(G)}}\lambda(B).
\]
In the coprime regime $(|A|,|K|)=1$, Schur--Zassenhaus and coprime fixed-point rigidity force
\[
\lambda(B)=|A:\C_A(B)|\qquad (B\le K),
\]
so the exact tuple count becomes purely fixed-point-theoretic. When $K$ is cyclic, say $K=C_\omega$, this specializes to the divisor-stratified Jordan-totient formula
\[
|\Hom(\mathbb Z^r,A\rtimes C_\omega)|=\sum_{n\mid \omega}|A:\C_A(B_n)|\,|\C_A(B_n)|^r J_r(n),
\]
where $J_r(n)$ is the $r$th Jordan totient and $B_n\le C_\omega$ is the unique subgroup of order $n$. Thus the cyclic case is the one-generator specialization in which the lift multiplicity collapses to the fixed-stratum index $|A:\C_A(B_n)|$.

The closest predecessors on the counting side are Torres--Giese, who proves that $P_2(G,s)$ is a finite Dirichlet series built from the poset of abelian subgroups \cite{TorresGiese2012}, and Kaur--Prajapati--Prasad, who prove that the orbit series $B_G(t)=\sum_{r\ge0}\kappa_r(G)t^r$ is rational with exponential base $m(G)$ \cite{KaurPrajapatiPrasad2019}. The present paper sharpens that picture by identifying the exact leading coefficient $N_{\max}(G)$, the secondary scale $b(G)$, and the first pole of $\mathcal P_G(z)$ together with the coefficient of its leading term. The rigidity, extremal, and low-rank formula aspects are treated separately in the companion paper \cite{LevitShwartzPartII}.

Sections~2 and~3 develop the basic properties of $P_r(G)$ and the centralizer/orbit recursions. Section~4 proves the dominant asymptotic, the entropy corollaries, the finite Dirichlet-spectrum theorem, the recurrence and Hankel-rank consequences, the inverse-spectrum rigidity theorem, and the exact formulas for abelian split quotients. Section~5 concludes.
\section{Definitions and basic properties}

\begin{definition}[Higher commutativity probabilities]
Let $G$ be a finite group and $r\ge 1$ an integer. Define
\[
\Comm_r(G):=\{(x_1,\dots,x_r)\in G^r : [x_i,x_j]=1\ \forall i,j\},
\]
and
\[
P_r(G):=\frac{|\Comm_r(G)|}{|G|^r}.
\]
Equivalently, $\Comm_r(G)=\Hom(\mathbb Z^r,G)$ and $P_r(G)=|\Hom(\mathbb Z^r,G)|/|G|^r$.
\end{definition}

\begin{remark}
$P_1(G)=1$ for every finite group $G$, and $P_r(G)=1$ for some $r\ge 2$ if and only if $G$ is abelian.
For $r=2$, $P_2(G)$ is the classical commuting probability (a.k.a.\ commutativity degree).
For $r=3$, $P_3(G)$ is the probability that three random elements generate an abelian subgroup.
\end{remark}

\begin{remark}[Relation to multiple commutativity degree]\label{rem:multiple-degree}
In the notation of Lescot \cite{Lescot1995} and Rezaei--Niroomand--Erfanian \cite{RezaeiNiroomandErfanian2014},
the \emph{multiple commutativity degree} $d_n(G)$ is the probability that $n{+}1$ uniform elements of $G$ commute pairwise.
Thus $P_r(G)=d_{r-1}(G)$ for every $r\ge 2$.
\end{remark}

\subsection{Isoclinism and invariance of \texorpdfstring{$P_r$}{Pr}}\label{subsec:isoclinism}

It is natural to regard the family $\{P_r(G)\}_{r\ge 2}$ as a ``commutator-geometry'' invariant of $G$.
This is made precise by \emph{isoclinism}, introduced by P.~Hall \cite{Hall1940} and used systematically by
Lescot in the study of commuting probabilities \cite{Lescot1995}.

\begin{definition}[Isoclinism \cite{Hall1940}]\label{def:isoclinism}
Two groups $G,H$ are \emph{isoclinic} if there exist isomorphisms
\[
\phi:G/\Z(G)\xrightarrow{\ \sim\ } H/\Z(H),
\qquad
\psi:G'\xrightarrow{\ \sim\ } H',
\]
such that the commutator maps are intertwined:
\[
\psi\big([g_1,g_2]\big)=\big[\widehat{\phi}(g_1\Z(G)),\,\widehat{\phi}(g_2\Z(G))\big]
\quad\text{for all }g_1,g_2\in G,
\]
where $\widehat{\phi}(g\Z(G))$ denotes any lift of $\phi(g\Z(G))$ to $H$.
Equivalently, if we write
\[
\kappa_G:(G/\Z(G))^2\to G',\qquad (g\Z(G),h\Z(G))\mapsto [g,h],
\]
(and similarly $\kappa_H$), then $\psi\circ \kappa_G=\kappa_H\circ(\phi\times\phi)$.
\end{definition}

\begin{proposition}[Isoclinism invariance of higher commuting probabilities]\label{prop:isoclinism-invariant}
If $G$ and $H$ are isoclinic finite groups, then
\[
P_r(G)=P_r(H)\qquad\text{for all }r\ge 2.
\]
In particular, the entire sequence $\{P_r(G)\}_{r\ge 2}$ depends only on the isoclinism class of $G$
(a result already emphasized by Lescot \cite{Lescot1995} in the language of multiple commutativity degrees).
\end{proposition}

\begin{proof}
Write $V_G:=G/\Z(G)$ and $W_G:=G'$.
For $r\ge 2$ consider the map
\[
\Phi_r^G:V_G^r\to W_G^{\binom r2},\qquad
(v_1,\dots,v_r)\mapsto \big(\kappa_G(v_i,v_j)\big)_{1\le i<j\le r}.
\]
The condition that $(x_1,\dots,x_r)\in G^r$ commute pairwise depends only on the cosets $\bar x_i:=x_i\Z(G)\in V_G$ and is equivalent to the condition
\[\Phi_r^G(\bar x_1,\dots,\bar x_r)=1.\]
Moreover, once a commuting $r$-tuple of cosets is fixed, there are exactly $|\Z(G)|^r$ lifts to commuting
$r$-tuples in $G^r$ (multiplying any coordinate by a central element preserves commutativity).
Hence
\[
|\Comm_r(G)|=|\Z(G)|^r\cdot |\ker(\Phi_r^G)|,
\qquad\text{so}\qquad
P_r(G)=\frac{|\ker(\Phi_r^G)|}{|V_G|^r}.
\]

Now suppose $G$ and $H$ are isoclinic via $(\phi,\psi)$ as in Definition~\ref{def:isoclinism}.
Then the intertwining condition implies that $\phi^r:V_G^r\to V_H^r$ carries
$\ker(\Phi_r^G)$ bijectively onto $\ker(\Phi_r^H)$.
Since $\phi$ is an isomorphism, $|V_G|=|V_H|$ and $|\ker(\Phi_r^G)|=|\ker(\Phi_r^H)|$.
Thus $P_r(G)=P_r(H)$.
\end{proof}

\begin{proposition}[Maximal and maximum abelian subgroups under isoclinism]\label{prop:isoclinism-max-abelian}
Let $G,H$ be isoclinic finite groups via $(\phi,\psi)$ as in Definition~\ref{def:isoclinism}, and let
$\pi_G:G\to G/\Z(G)$ and $\pi_H:H\to H/\Z(H)$ be the quotient maps.
For every subgroup $A$ with $\Z(G)\le A\le G$ define
\[
A^\phi:=\pi_H^{-1}\bigl(\phi(\pi_G(A))\bigr)\le H.
\]
Then $A\mapsto A^\phi$ is an inclusion-preserving bijection between the subgroups of $G$ containing $\Z(G)$
and the subgroups of $H$ containing $\Z(H)$, and it preserves quotient sizes:
\[
|A:\Z(G)|=|A^\phi:\Z(H)|.
\]
Moreover, $A$ is abelian if and only if $A^\phi$ is abelian.

Consequently, $G$ and $H$ have the same number of:
\begin{enumerate}
\item maximal abelian subgroups (maximal by inclusion);
\item abelian subgroups of maximum order.
\end{enumerate}
In particular, if $A$ is abelian of maximum order in $G$ and $B=A^\phi$, then $B$ is abelian of maximum
order in $H$ and $|G:A|=|H:B|$.
\end{proposition}

\begin{proof}
Since $\phi$ is an isomorphism, $A\mapsto A^\phi$ is a bijection between subgroups containing the centers,
with inverse given by $\phi^{-1}$; it is inclusion-preserving and
$|A/\Z(G)|=|A^\phi/\Z(H)|$.
To check abelianness, note that $A$ is abelian if and only if the commutator map $\kappa_G$ from
Definition~\ref{def:isoclinism} vanishes on $\pi_G(A)\times \pi_G(A)$.
By isoclinism, $\kappa_H$ vanishes on $\phi(\pi_G(A))\times \phi(\pi_G(A))$, which is equivalent to $A^\phi$ being abelian.

Every maximal abelian subgroup contains the center: if $A$ is abelian and $z\in \Z(G)\setminus A$, then $\langle A,z\rangle$ is abelian and strictly larger.
Hence maximality by inclusion is preserved by the inclusion-preserving bijection.

Similarly, every abelian subgroup of maximum order contains the center, and maximizing $|A|$ is equivalent to maximizing
$|A/\Z(G)|$ among abelian subgroups containing $\Z(G)$; this is preserved by the size-preserving bijection.
Finally, if $A$ corresponds to $B$, then
\[
|G:A|=\frac{|G/\Z(G)|}{|A/\Z(G)|}=\frac{|H/\Z(H)|}{|B/\Z(H)|}=|H:B|.
\]
\end{proof}

\begin{proposition}[Maximum-order abelian subgroups and higher commuting probabilities: the sharp intersection case]\label{prop:maxabelian-Pr-intersection}
Let $G$ be a finite group. Let $\mathcal M(G)$ be the set of maximal abelian subgroups of $G$, and set
\[
M(G):=|\mathcal M(G)|,\qquad z:=|Z(G)|.
\]
Let $\mathcal A_{\max}(G)$ be the set of maximum-order abelian subgroups of $G$.
Assume that distinct members of $\mathcal A_{\max}(G)$ intersect
exactly in the center:
\[
A,B\in\mathcal A_{\max}(G),\ A\neq B\quad\Longrightarrow\quad A\cap B=Z(G).
\]
Then for every $r\ge 1$,
\begin{equation}\label{eq:Pr-maxabelian-sandwich}
\frac{N_{\max}(G)\,m(G)^r-\bigl(N_{\max}(G)-1\bigr)z^r}{|G|^r}\ \le\ P_r(G)\ \le\ \frac{M(G)\,m(G)^r}{|G|^r}.
\end{equation}
\end{proposition}

\begin{proof}
Every commuting $r$-tuple lies in some maximal abelian subgroup, hence
\[
\Comm_r(G)\subseteq \bigcup_{A\in\mathcal M(G)}A^r,
\]
so $|\Comm_r(G)|\le \sum_{A\in\mathcal M(G)}|A|^r\le M(G)\,m(G)^r$, giving the upper bound in \eqref{eq:Pr-maxabelian-sandwich}.
For the lower bound, note that $\bigcup_{A\in\mathcal A_{\max}(G)}A^r\subseteq \Comm_r(G)$.
Under the hypothesis $A\cap B=Z(G)$ for distinct $A,B\in\mathcal A_{\max}(G)$, the union has size
\[
\Bigl|\bigcup_{A\in\mathcal A_{\max}(G)}A^r\Bigr|
= N_{\max}(G)\,m(G)^r-\bigl(N_{\max}(G)-1\bigr)z^r,
\]
since all multiple intersections coincide with $Z(G)^r$.
Dividing by $|G|^r$ yields the lower bound.
\end{proof}

\begin{corollary}[Extraspecial $p$-groups: maximum-order abelian incidence bounds $P_r$]\label{cor:extraspecial-maxabelian}
Let $G$ be an extraspecial $p$-group of order $|G|=p^{2n+1}$, so $Z(G)=G'$ has order $p$ and
$V:=G/Z(G)$ is an $\mathbb{F}_p$-vector space of dimension $2n$ equipped with the induced nondegenerate alternating form.
Then every abelian subgroup of maximum order has order
\[
m(G)=p^{n+1},
\]
and the set $\mathcal A_{\max}(G)$ of abelian subgroups of order $m(G)$ coincides with the set of maximal abelian
subgroups.
Moreover, $\mathcal A_{\max}(G)$ is in bijection with the set of Lagrangian (maximal totally isotropic) subspaces of $V$, hence
\[
N_{\max}(G)=M(G)=\prod_{i=1}^{n}(p^{i}+1).
\]
Finally, if $A\neq B$ are maximal abelian in $G$, then
\[
|A\cap B|\le p^{n}.
\]

Consequently, for every $r\ge 1$,
\[
\frac{N_{\max}(G)\,p^{(n+1)r}-\binom{N_{\max}(G)}{2}p^{nr}}{p^{(2n+1)r}}
\ \le\ P_r(G)\ \le\
\frac{N_{\max}(G)\,p^{(n+1)r}}{p^{(2n+1)r}},
\]
and in particular
\[
P_r(G)=N_{\max}(G)\,p^{-nr}+O\!\left(p^{-(n+1)r}\right)\qquad (r\to\infty).
\]
\end{corollary}

\begin{proof}
It is standard that in an extraspecial $p$-group one has $Z(G)=G'$ of order $p$ and $G/Z(G)$ is elementary abelian of
order $p^{2n}$, with commutator inducing a nondegenerate alternating form on $V$.
Abelian subgroups containing $Z(G)$ correspond to totally isotropic subspaces of $V$, and maximal abelian subgroups
correspond to Lagrangian subspaces. In particular, maximal abelian subgroups have order $p^{n+1}$.

The symplectic group $\mathrm{Sp}_{2n}(p)$ acts transitively on the Lagrangian subspaces of $V$.
The stabilizer of one Lagrangian is the corresponding maximal parabolic subgroup, of order
$p^{n(n+1)/2}|\mathrm{GL}_n(p)|=p^{n^2}\prod_{i=1}^{n}(p^i-1)$, while
$|\mathrm{Sp}_{2n}(p)|=p^{n^2}\prod_{i=1}^{n}(p^{2i}-1)$.
Hence the number of Lagrangian subspaces is
\[
\frac{|\mathrm{Sp}_{2n}(p)|}{p^{n(n+1)/2}|\mathrm{GL}_n(p)|}
=\prod_{i=1}^{n}\frac{p^{2i}-1}{p^i-1}
=\prod_{i=1}^{n}(p^{i}+1),
\]
and therefore $N_{\max}(G)=M(G)$.
If $A\neq B$ are maximal abelian then $A/Z(G)$ and $B/Z(G)$ are distinct Lagrangians, so
$\dim(A\cap B)/Z(G)\le n-1$ and therefore $|A\cap B|\le p^{n}$.

For the upper bound, every commuting $r$-tuple lies in some maximal abelian subgroup, so
$|\Comm_r(G)|\le \sum_{A\in\mathcal A_{\max}(G)}|A|^r=N_{\max}(G)\,p^{(n+1)r}$.
For the lower bound, the union $\bigcup_{A\in\mathcal A_{\max}(G)}A^r$ is contained in $\Comm_r(G)$.
By the (second) Bonferroni inequality,
\[
\Bigl|\bigcup_{A\in\mathcal A_{\max}(G)}A^r\Bigr|\ge
\sum_{A\in\mathcal A_{\max}(G)}|A|^r-\!\!\sum_{\substack{A,B\in\mathcal A_{\max}(G)\\A<B}}\!|(A\cap B)|^r
\ \ge\ N_{\max}(G)\,p^{(n+1)r}-\binom{N_{\max}(G)}{2}p^{nr}.
\]
Dividing by $|G|^r=p^{(2n+1)r}$ yields the displayed bounds. The asymptotic follows since the error term is
$O(p^{-(n+1)r})$ as $r\to\infty$.
\end{proof}

\begin{proposition}[Monotonicity under quotients]\label{prop:quotient-monotone}
Let $N\trianglelefteq G$ be a normal subgroup. Then for every $r\ge 1$,
\[
P_r(G/N)\ \ge\ P_r(G).
\]
Equivalently, passing to a quotient can only increase the probability that $r$ random elements commute.
\end{proposition}

\begin{proof}
Let $\pi:G\to G/N$ be the quotient map.
A commuting $r$-tuple in $G$ maps to a commuting $r$-tuple in $G/N$.
Choose $(x_1,\dots,x_r)\in G^r$ uniformly at random.
The induced tuple in $(G/N)^r$ is again uniform, and
\[
\begin{aligned}
P_r(G/N)
&=\Prob([\pi(x_i),\pi(x_j)]=1\ \forall i,j)\\
&=\Prob([x_i,x_j]\in N\ \forall i,j)\\
&\ge \Prob([x_i,x_j]=1\ \forall i,j)=P_r(G).
\end{aligned}
\]
\end{proof}

\begin{remark}\label{rem:kappa-not-isoclinism}
In contrast, the unnormalized orbit counts $\kappa_r(G)=|G|^rP_{r+1}(G)$ are generally \emph{not}
isoclinism invariants, since isoclinic groups can have different center sizes and hence different orders.
\end{remark}

\begin{proposition}[Monotonicity in $r$]\label{prop:monotone}
For every finite group $G$ and every $r\ge 2$,
\[
P_{r+1}(G)\le P_r(G).
\]
Moreover, if $G$ is non-abelian, then the inequality is strict for all $r\ge 2$.
\end{proposition}

\begin{proof}
Choose $(x_1,\dots,x_{r+1})\in G^{r+1}$ uniformly at random.
The event ``$(x_1,\dots,x_{r+1})$ commute pairwise'' implies the event
``$(x_1,\dots,x_r)$ commute pairwise'', hence $P_{r+1}(G)\le P_r(G)$.

Assume $G$ is non-abelian and pick $a\in G\setminus \Z(G)$.
Let $E$ be the event that $(x_1,\dots,x_r)=(a,1,\dots,1)$.
Then $E\subseteq \Comm_r(G)$ and $\Prob(E)=|G|^{-r}>0$.
Conditioned on $E$, the probability that $x_{r+1}$ commutes with all previous coordinates is
$|\C_G(a)|/|G|<1$.
Therefore
\[
P_{r+1}(G)=\Prob((x_1,\dots,x_{r+1})\in\Comm_{r+1}(G))
\le P_r(G)-\Prob(E)\Bigl(1-\frac{|\C_G(a)|}{|G|}\Bigr)
< P_r(G),
\]
so the inequality is strict.
\end{proof}

\begin{remark}[Unimodality is automatic]\label{rem:unimodality}
By Proposition~\ref{prop:monotone}, the sequence $\{P_r(G)\}_{r\ge 1}$ is weakly decreasing, hence unimodal, with its mode at $r=1$.
If $G$ is non-abelian, then it is strictly decreasing.
Equivalently, the shifted sequence $d_n(G)=P_{n+1}(G)$ is always unimodal.
In particular, every truncation
\[
\sum_{r=2}^{N} P_r(G)\,z^{r-2}
\]
has unimodal coefficients.
Thus the genuinely delicate shape question is not unimodality but true log-concavity.
\end{remark}

\begin{proposition}[Direct products]\label{prop:direct}
If $G,H$ are finite groups, then for every $r\ge 1$,
\[
P_r(G\times H)=P_r(G)\,P_r(H).
\]
\end{proposition}

\begin{proof}
A tuple in $(G\times H)^r$ commutes pairwise if and only if its $G$-coordinates commute
pairwise and its $H$-coordinates commute pairwise. Thus
$\Comm_r(G\times H)\cong \Comm_r(G)\times \Comm_r(H)$ and the formula follows after dividing
by $|G\times H|^r=|G|^r|H|^r$.
\end{proof}

\begin{proposition}[Submultiplicativity across indices]\label{prop:submult}
For every finite group $G$ and all integers $n,m\ge 1$,
\[
P_{n+m}(G)\le P_n(G)\,P_m(G).
\]
Equivalently, the sequence $a_r:=-\log P_r(G)$ is superadditive: $a_{n+m}\ge a_n+a_m$.

Moreover, if $G$ is non-abelian then the inequality is strict for all $n,m\ge 1$.
Equivalently, equality in $P_{n+m}(G)\le P_n(G)P_m(G)$ for some $n,m\ge 1$ forces $G$ abelian.
\end{proposition}

\begin{proof}
Let $(x_1,\dots,x_{n+m})$ be uniformly random in $G^{n+m}$.
Let $A$ be the event that $(x_1,\dots,x_n)\in \Comm_n(G)$ and $B$ the event that
$(x_{n+1},\dots,x_{n+m})\in \Comm_m(G)$.
Then $\Comm_{n+m}(G)\subseteq A\cap B$.
Since $A$ depends only on the first $n$ coordinates and $B$ only on the last $m$, the events $A$ and $B$
are independent. Hence
\[
P_{n+m}(G)=\Prob(\Comm_{n+m}(G))\le \Prob(A\cap B)=\Prob(A)\Prob(B)=P_n(G)\,P_m(G).
\]

If $G$ is non-abelian choose $a,b\in G$ with $[a,b]\neq 1$.
Consider the event
\[
E=\{(x_1,\dots,x_{n+m})=(a,1,\dots,1,b,1,\dots,1)\}.
\]
Then $E\subseteq A\cap B$ but $E\cap \Comm_{n+m}(G)=\varnothing$, since $a$ and $b$ do not commute.
Moreover $\Prob(E)=|G|^{-(n+m)}$, so
\[
P_{n+m}(G)=\Prob(\Comm_{n+m}(G))\le \Prob(A\cap B)-\Prob(E)=P_n(G)P_m(G)-|G|^{-(n+m)}<P_n(G)P_m(G).
\]
\end{proof}

\begin{remark}[An asymptotic decay rate]\label{rem:decay-rate}
By Proposition~\ref{prop:submult}, the sequence $a_r:=-\log P_r(G)\in[0,\infty]$ is superadditive, so by
Fekete's lemma the limit
\[
\lim_{r\to\infty}\frac{a_r}{r}=\sup_{r\ge 1}\frac{a_r}{r}
\]
exists. Equivalently, the limit $\lim_{r\to\infty} P_r(G)^{1/r}$ exists and equals
$\inf_{r\ge 1} P_r(G)^{1/r}$.
\end{remark}

\section{Centralizer recursion and orbit counts}

\subsection{A recursion via centralizers}
For $x\in G$, let $\C_G(x)=\{g\in G:gx=xg\}$ denote the centralizer.

\begin{lemma}[Centralizer recursion]\label{lem:recursion}
For every finite group $G$ and every $r\ge 2$,
\begin{equation}\label{Pr-formula}
P_r(G)=\frac{1}{|G|^r}\sum_{x\in G} |\C_G(x)|^{r-1}\,P_{r-1}(\C_G(x)).
\end{equation}
\end{lemma}

\begin{proof}
Count $\Comm_r(G)$ by the first coordinate $x_1=x$. Given $x\in G$, an $r$-tuple
$(x,x_2,\dots,x_r)$ commutes pairwise if and only if $(x_2,\dots,x_r)\in \Comm_{r-1}(\C_G(x))$.
Hence
\[
|\Comm_r(G)|=\sum_{x\in G} |\Comm_{r-1}(\C_G(x))|
=\sum_{x\in G} |\C_G(x)|^{r-1} P_{r-1}(\C_G(x)).
\]
Then dividing by $|G|^r$, we get \eqref{Pr-formula}.
\end{proof}

\subsection{Orbit-count normalization and integrality}

\begin{definition}[Higher class numbers]\label{def:kappa}
For $r\ge 0$ define
\[
\kappa_r(G):=\big|\Hom(\mathbb Z^r,G)/G\big|,
\]
the number of diagonal conjugacy orbits of commuting $r$-tuples.
By convention, $\kappa_0(G)=1$ for every $G$, and $\kappa_1(G)=k(G)$ is the number of conjugacy classes of $G$.
\end{definition}

\begin{proposition}[Monotonicity of the orbit counts]\label{prop:kappa-monotone}
For every finite group $G$ and every $r\ge 1$,
\[
\kappa_{r+1}(G)\ge \kappa_r(G).
\]
Moreover, if $|G|>1$ then the inequality is strict.
\end{proposition}

\begin{proof}
Consider the map
\[
\iota:\Hom(\mathbb Z^r,G)\to \Hom(\mathbb Z^{r+1},G),\qquad (x_1,\dots,x_r)\mapsto (x_1,\dots,x_r,1).
\]
It is $G$-equivariant for diagonal conjugation, hence induces a map of orbit sets
$\bar\iota:\Hom(\mathbb Z^r,G)/G\to \Hom(\mathbb Z^{r+1},G)/G$.
If $\bar\iota([\mathbf x])=\bar\iota([\mathbf y])$ then there is $g\in G$ with
$gx_i g^{-1}=y_i$ for $1\le i\le r$ (and automatically $g1g^{-1}=1$), so $[\mathbf x]=[\mathbf y]$.
Thus $\bar\iota$ is injective, proving $\kappa_{r+1}(G)\ge \kappa_r(G)$.

If $|G|>1$, pick $h\in G\setminus\{1\}$. The commuting $(r{+}1)$-tuple $(1,\dots,1,h)$ defines a
$G$-orbit in $\Hom(\mathbb Z^{r+1},G)$ that is not in the image of $\bar\iota$ (since the last coordinate
of every tuple in the image equals $1$). Hence the inequality is strict.
\end{proof}

\begin{theorem}[Burnside orbit lemma normalization]\label{thm:burnside}
For every finite group $G$ and every $r\ge 1$,
\[
\kappa_r(G)=|G|^r\,P_{r+1}(G)
\qquad\text{equivalently}\qquad
P_{r+1}(G)=\frac{\kappa_r(G)}{|G|^r}.
\]
In particular, $|G|^r P_{r+1}(G)$ is always an integer.
\end{theorem}

\begin{proof}
Let $X:=\Hom(\mathbb Z^r,G)=\Comm_r(G)$ and let $G$ act on $X$ by diagonal conjugation:
$h\cdot(x_1,\dots,x_r)=(hx_1h^{-1},\dots,hx_rh^{-1})$.
By Burnside's lemma,
\[
|X/G|=\frac{1}{|G|}\sum_{h\in G} |X^h|,
\]
where $X^h$ is the fixed-point set of $h$. A tuple $(x_1,\dots,x_r)$ is fixed by $h$
iff each $x_i\in \C_G(h)$, and it still must commute pairwise; thus $X^h=\Comm_r(\C_G(h))$
and $|X^h|=|\C_G(h)|^r P_r(\C_G(h))$.
Therefore
\[
\kappa_r(G)=\frac{1}{|G|}\sum_{h\in G}|\C_G(h)|^r P_r(\C_G(h)).
\]
Comparing with Lemma~\ref{lem:recursion} (applied with $r+1$ in place of $r$) gives
$\kappa_r(G)=|G|^r P_{r+1}(G)$.
\end{proof}

\begin{proposition}[A congruence for $p$-groups]\label{prop:pgroup-orbit-congruence}
Let $G$ be a finite $p$-group and let $r\ge 0$.
Then
\[
\kappa_r(G)=\big|\Hom(\mathbb Z^r,G)/G\big|\equiv |\Z(G)|^r \pmod p.
\]
Equivalently,
\[
|G|^r P_{r+1}(G)\equiv |\Z(G)|^r \pmod p.
\]
\end{proposition}

\begin{proof}
A finite $p$-group acting on a finite set has all orbit sizes powers of $p$.
Hence the number of orbits is congruent modulo $p$ to the number of fixed points.
Apply this to the diagonal conjugation action of $G$ on $\Hom(\mathbb Z^r,G)=\Comm_r(G)$:
a commuting $r$-tuple is fixed by all of $G$ if and only if each coordinate lies in $\Z(G)$.
Thus the fixed-point set has size $|\Z(G)|^r$, proving the congruence.
\end{proof}

\begin{proposition}[$\kappa$-recursion by centralizers]\label{prop:kappa-recursion}
For every finite group $G$ and every $r\ge 1$,
\[
\kappa_r(G)=\sum_{[g]\subseteq G}\kappa_{r-1}(\C_G(g)),
\]
where the sum runs over conjugacy classes in $G$.
In particular, for $r=1$ this says $k(G)=\sum_{[g]}1$, and for $r=2$ it becomes
\[
\kappa_2(G)=\sum_{[g]\subseteq G} k(\C_G(g)).
\]
\end{proposition}

\begin{proof}
Combine Theorem~\ref{thm:burnside} with Lemma~\ref{lem:recursion}, grouping the sum by conjugacy classes.
\end{proof}

\begin{proposition}[Direct products for $\kappa_r$]\label{prop:kappa-direct}
If $G,H$ are finite groups, then for every $r\ge 0$,
\[
\kappa_r(G\times H)=\kappa_r(G)\,\kappa_r(H).
\]
In particular, if $A$ is abelian then
\[
\kappa_r(G\times A)=|A|^r\,\kappa_r(G),
\qquad\text{while}\qquad
P_r(G\times A)=P_r(G).
\]
\end{proposition}

\begin{proof}
Use $\Hom(\mathbb Z^r,G\times H)\cong \Hom(\mathbb Z^r,G)\times\Hom(\mathbb Z^r,H)$ and the fact that
diagonal conjugation by $G\times H$ acts separately on the two factors. The final claim follows from
Proposition~\ref{prop:direct} and the fact that $P_r(A)=1$ for abelian $A$.
\end{proof}

\section{Asymptotics and finite spectrum}

\subsection{Asymptotics, entropy, and generating series}\label{subsec:dominant-abelian}

The root-limit statement from Remark~\ref{rem:decay-rate} can in fact be sharpened to an exact leading term controlled by the maximum-order abelian subgroups of $G$.

\begin{theorem}[Dominant abelian asymptotic]\label{thm:dominant-abelian}
Let $G$ be a finite group. With the notation $m(G)$, $N_{\max}(G)$, and $b(G)$ fixed in the introduction, one has, as $r\to\infty$,
\begin{equation}\label{eq:dominant-Comm}
|\Hom(\mathbb Z^r,G)|
=
N_{\max}(G)\,m(G)^r+O\!\bigl(b(G)^r\bigr).
\end{equation}
Equivalently,
\begin{equation}\label{eq:dominant-Pr}
P_r(G)
=
N_{\max}(G)\Bigl(\frac{m(G)}{|G|}\Bigr)^r
+
O\!\Bigl(\Bigl(\frac{b(G)}{|G|}\Bigr)^r\Bigr).
\end{equation}
Moreover,
\begin{equation}\label{eq:dominant-kappa}
\kappa_r(G)
=
\frac{N_{\max}(G)\,m(G)}{|G|}\,m(G)^r+O\!\bigl(b(G)^r\bigr).
\end{equation}
In particular,
\[
\lim_{r\to\infty}P_r(G)^{1/r}=\frac{m(G)}{|G|},
\qquad
\lim_{r\to\infty}\kappa_r(G)^{1/r}=m(G).
\]
\end{theorem}

\begin{proof}
If $G$ is abelian, then $m(G)=|G|$, $N_{\max}(G)=1$, $b(G)=0$, and the formulas are immediate. Assume henceforth that $G$ is non-abelian.

Let $\mathcal M(G)$ be the set of maximal abelian subgroups of $G$, and let
\[
\mathcal A_{\max}(G):=\{A\in \mathcal M(G): |A|=m(G)\}
=\{A\le G:\ A\text{ abelian},\ |A|=m(G)\}.
\]
Every commuting $r$-tuple lies in some maximal abelian subgroup, hence
\[
\Hom(\mathbb Z^r,G)=\Comm_r(G)=\bigcup_{M\in\mathcal M(G)} M^r.
\]
Write
\[
U_r:=\bigcup_{A\in \mathcal A_{\max}(G)} A^r,
\qquad
V_r:=\bigcup_{M\in \mathcal M(G)\setminus \mathcal A_{\max}(G)} M^r.
\]
Then $\Comm_r(G)=U_r\cup V_r$.

If $A_1,\dots,A_t\in \mathcal A_{\max}(G)$ are distinct, then $A_1\cap \cdots \cap A_t$ is abelian and has order $<m(G)$, hence at most $b(G)$.
By inclusion--exclusion,
\[
|U_r|
=
N_{\max}(G)\,m(G)^r
+
\sum_{t=2}^{N_{\max}(G)} (-1)^{t+1}
\sum_{1\le i_1<\cdots < i_t\le N_{\max}(G)}
|A_{i_1}\cap\cdots\cap A_{i_t}|^r,
\]
so
\[
|U_r|=N_{\max}(G)\,m(G)^r+O\!\bigl(b(G)^r\bigr).
\]
On the other hand, every subgroup in $\mathcal M(G)\setminus \mathcal A_{\max}(G)$ has order at most $b(G)$, so
\[
|V_r|\le |\mathcal M(G)|\,b(G)^r = O\!\bigl(b(G)^r\bigr).
\]
Since $\Comm_r(G)=U_r\cup V_r$, formula \eqref{eq:dominant-Comm} follows.
Dividing by $|G|^r$ gives \eqref{eq:dominant-Pr}.
Finally, using
\[
\kappa_r(G)=|G|^r P_{r+1}(G)
\]
yields
\[
\kappa_r(G)
=
|G|^r\left(
N_{\max}(G)\Bigl(\frac{m(G)}{|G|}\Bigr)^{r+1}
+
O\!\Bigl(\Bigl(\frac{b(G)}{|G|}\Bigr)^{r+1}\Bigr)\right)
=
\frac{N_{\max}(G)m(G)}{|G|}\,m(G)^r+O\!\bigl(b(G)^r\bigr),
\]
which is \eqref{eq:dominant-kappa}. The two root limits follow immediately.
\end{proof}

\begin{corollary}[Two complementary entropy normalizations]\label{cor:entropy-bridge}
Define
\[
\lambda_{\mathrm{prob}}(G):=\lim_{r\to\infty}P_r(G)^{1/r},
\qquad
\lambda_{\mathrm{orb}}(G):=\lim_{r\to\infty}\kappa_r(G)^{1/r}.
\]
Then
\[
\lambda_{\mathrm{prob}}(G)=\frac{m(G)}{|G|},
\qquad
\lambda_{\mathrm{orb}}(G)=m(G).
\]
Equivalently, if
\[
h_{\mathrm{prob}}(G):=-\log \lambda_{\mathrm{prob}}(G),
\qquad
h_{\mathrm{orb}}(G):=\log \lambda_{\mathrm{orb}}(G),
\]
then
\[
h_{\mathrm{prob}}(G)=\log\frac{|G|}{m(G)},
\qquad
h_{\mathrm{orb}}(G)=\log m(G),
\qquad
h_{\mathrm{prob}}(G)+h_{\mathrm{orb}}(G)=\log |G|.
\]
\end{corollary}

\begin{proposition}[Split abelian extensions and fixed-subgroup geometry]\label{prop:semidirect-largest}
Let $G=A\rtimes K$ be a finite semidirect product with $A$ abelian. Then
\[
m(G)=\max_{B\le K\ \text{abelian}} |\C_A(B)|\,|B|.
\]
Consequently,
\[
\lambda_{\mathrm{prob}}(G)=\frac{1}{|A||K|}\max_{B\le K\ \text{abelian}} |\C_A(B)|\,|B|,
\qquad
\lambda_{\mathrm{orb}}(G)=\max_{B\le K\ \text{abelian}} |\C_A(B)|\,|B|,
\]
and equivalently
\[
h_{\mathrm{prob}}(G)=\log\frac{|A||K|}{\max_{B\le K\ \text{abelian}} |\C_A(B)|\,|B|},
\qquad
h_{\mathrm{orb}}(G)=\log\Bigl(\max_{B\le K\ \text{abelian}} |\C_A(B)|\,|B|\Bigr).
\]
\end{proposition}

\begin{proof}
Let $M\le G$ be abelian, and let $\pi:G\to K$ be the natural projection. Then $B:=\pi(M)$ is abelian.
The kernel of $\pi|_M$ is $M\cap A$, so
\[
|M|=|M\cap A|\,|B|.
\]
If $a\in M\cap A$ and $m\in M$ projects to $b\in B$, then $ma=am$ because $M$ is abelian, hence conjugation by $m$ fixes $a$.
Since conjugation by $m$ on $A$ agrees with the action of $b$, it follows that $a\in \C_A(b)$. Therefore
\[
M\cap A\le \C_A(B):=\bigcap_{b\in B}\C_A(b),
\]
so
\[
|M|\le |\C_A(B)|\,|B|.
\]
Taking the maximum over all abelian $M\le G$ gives
\[
m(G)\le \max_{B\le K\ \text{abelian}} |\C_A(B)|\,|B|.
\]
Conversely, if $B\le K$ is abelian, then $\C_A(B)\rtimes B$ is an abelian subgroup of $G$: the subgroup $B$ is abelian, the subgroup $\C_A(B)$ is abelian, and $B$ acts trivially on $\C_A(B)$. Thus
\[
m(G)\ge |\C_A(B)|\,|B|
\]
for every abelian $B\le K$, proving the formula for $m(G)$. The displayed formulas for $\lambda_{\mathrm{prob}}(G)$, $\lambda_{\mathrm{orb}}(G)$, $h_{\mathrm{prob}}(G)$, and $h_{\mathrm{orb}}(G)$ now follow from Corollary~\ref{cor:entropy-bridge}.
\end{proof}

\begin{theorem}[Abelian split quotients: exact tuple count]\label{thm:abelian-exact-count}
Let $G=A\rtimes K$ with $A$ and $K$ finite abelian, and let $\pi:G\to K$ be the natural projection.
For each subgroup $B\le K$ define
\[
C_B:=\C_A(B),
\qquad
\Lambda(B):=\{H\le G: H\text{ abelian},\ \pi(H)=B,\ H\cap A=C_B\},
\qquad
\lambda(B):=|\Lambda(B)|.
\]
Thus $\Lambda(B)$ is the family of abelian lifts of $B$ whose kernel in $A$ is exactly $C_B$.
Also set
\[
\varphi_r(B):=\#\{(b_1,\dots,b_r)\in B^r:\ \langle b_1,\dots,b_r\rangle=B\}.
\]
Then for every $r\ge 1$,
\[
|\Hom(\mathbb Z^r,G)|=\sum_{B\le K}\lambda(B)\,|C_B|^r\,\varphi_r(B).
\]
Equivalently, if $\mu_B(D):=\mu(D,B)$ denotes the M\"obius function of the subgroup lattice interval $[D,B]$, then
\[
\varphi_r(B)=\sum_{D\le B}\mu_B(D)\,|D|^r
\]
and therefore
\[
|\Hom(\mathbb Z^r,G)|=\sum_{B\le K}\sum_{D\le B}\lambda(B)\mu_B(D)\,(|C_B|\,|D|)^r.
\]
\end{theorem}

\begin{proof}
For each subgroup $B\le K$ let
\[
\mathcal T_B(r):=\{(x_1,\dots,x_r)\in \Hom(\mathbb Z^r,G):\ \langle \pi(x_1),\dots,\pi(x_r)\rangle=B\}.
\]
Then the sets $\mathcal T_B(r)$, for $B\le K$, partition $\Hom(\mathbb Z^r,G)$.

Fix $\mathbf x=(x_1,\dots,x_r)\in \mathcal T_B(r)$ and let $H:=\langle x_1,\dots,x_r\rangle$.
Then $H$ is abelian and $\pi(H)=B$.
If $a\in H\cap A$ and $h\in H$ projects to $b\in B$, then $ha=ah$ because $H$ is abelian, hence conjugation by $h$ fixes $a$.
Since conjugation by $h$ on $A$ agrees with the action of $b$, it follows that $a\in C_B$.
Thus $H\cap A\le C_B$.
Define
\[
\widehat H:=C_BH\le G.
\]
Because $C_B$ centralizes every element of $H$, the subgroup $\widehat H$ is abelian.
Also $\pi(\widehat H)=B$ and
\[
\widehat H\cap A=C_B(H\cap A)=C_B,
\]
so $\widehat H\in \Lambda(B)$.
This proves existence of a subgroup in $\Lambda(B)$ containing the tuple $\mathbf x$.

We claim that this subgroup is unique.
Suppose $\mathbf x$ lies in both $H_1,H_2\in \Lambda(B)$, and let again $H:=\langle x_1,\dots,x_r\rangle$.
Since $\pi(H)=B$, for each $j\in\{1,2\}$ and each $y\in H_j$ there exists $h\in H$ with $\pi(h)=\pi(y)$.
Then $yh^{-1}\in H_j\cap A=C_B$, so $y\in C_BH$.
Therefore $H_j=C_BH$ for $j=1,2$, proving $H_1=H_2$.

Fix now one subgroup $H\in \Lambda(B)$.
The projection map $H\to B$ is surjective with kernel $C_B$.
Hence for each generating $r$-tuple in $B^r$ there are exactly $|C_B|^r$ lifts to an $r$-tuple in $H^r$.
So the number of tuples in $H^r$ whose projection generates $B$ is
\[
|C_B|^r\,\varphi_r(B).
\]
Summing over the $\lambda(B)$ distinct subgroups in $\Lambda(B)$ gives
\[
|\mathcal T_B(r)|=\lambda(B)\,|C_B|^r\,\varphi_r(B).
\]
Finally, summing over all $B\le K$ proves the first displayed formula.

For the M\"obius inversion formula, note that every $r$-tuple in $B^r$ generates a unique subgroup $D\le B$.
Therefore
\[
|B|^r=\sum_{D\le B}\varphi_r(D).
\]
M\"obius inversion on the subgroup lattice of $B$ yields
\[
\varphi_r(B)=\sum_{D\le B}\mu_B(D)\,|D|^r.
\]
Substituting this into the first formula gives the second displayed expansion.
\end{proof}

\begin{theorem}[Coprime lift multiplicities]\label{thm:coprime-lambda}
Let $G=A\rtimes K$ with $A$ and $K$ finite abelian, and retain the notation
\[
C_B:=\C_A(B),
\qquad
\Lambda(B):=\{H\le G:H\text{ abelian},\ \pi(H)=B,\ H\cap A=C_B\},
\qquad
\lambda(B):=|\Lambda(B)|
\]
from Theorem~\ref{thm:abelian-exact-count}. Thus $\Lambda(B)$ is the family of abelian lifts of $B$ whose kernel in $A$ is exactly $C_B$. If $B\le K$ satisfies
\[
\gcd(|A|,|B|)=1,
\]
then
\[
\lambda(B)=|A:C_B|.
\]
Equivalently, the full-kernel abelian lifts of $B$ are exactly the preimages of the complements of $A/C_B$ in $(A/C_B)\rtimes B$, and there are $|A:C_B|$ of them.
\end{theorem}

\begin{proof}
Set
\[
\overline A:=A/C_B,
\qquad
\overline G:=\overline A\rtimes B.
\]
If $H\in \Lambda(B)$, then $\overline H:=H/C_B$ is a complement to $\overline A$ in $\overline G$: indeed $\overline G=\overline A\,\overline H$ and
\[
\overline A\cap \overline H=(A/C_B)\cap (H/C_B)=(H\cap A)/C_B=1.
\]
Conversely, if $L\le \overline G$ is a complement to $\overline A$, let $\widetilde L$ be its full preimage in $A\rtimes B$. Then $\pi(\widetilde L)=B$ and
\[
\widetilde L\cap A=C_B.
\]
Also $|\widetilde L:C_B|=|L|=|B|$, and since $\gcd(|C_B|,|B|)=1$, Schur--Zassenhaus gives a complement $L_0\le \widetilde L$ to $C_B$. Because $L_0$ projects isomorphically onto $B$ and $B$ acts trivially on $C_B$, one has $[L_0,C_B]=1$. Hence
\[
\widetilde L=C_B\times L_0,
\]
so $\widetilde L$ is abelian. Thus $\widetilde L\in \Lambda(B)$. These two constructions are inverse to each other, so $\Lambda(B)$ is in bijection with the complements to $\overline A$ in $\overline G$.

Since $\gcd(|\overline A|,|B|)=1$, Schur--Zassenhaus implies that all complements to the normal Hall subgroup $\overline A$ are conjugate under $\overline A$. Hence the number of complements is the size of the $\overline A$-orbit of one complement, namely
\[
|\overline A:C_{\overline A}(B)|.
\]
It remains to show that $C_{\overline A}(B)=1$. Let $a+C_B\in \overline A$ be fixed by $B$. Then for each $b\in B$ one has
\[
\delta(b):=b\cdot a-a\in C_B.
\]
Because $B$ is abelian and acts trivially on $C_B$,
\[
\delta(bb')=bb'\cdot a-a=b\cdot(b'\cdot a-a)+(b\cdot a-a)=\delta(b')+\delta(b)
\]
for all $b,b'\in B$, so $\delta:B\to C_B$ is a homomorphism. But $\gcd(|B|,|C_B|)=1$, so every homomorphism $B\to C_B$ is trivial. Hence $\delta(b)=0$ for all $b$, which means $a\in C_B$. Therefore $a+C_B=0$, so $C_{\overline A}(B)=1$.

Thus the number of complements is
\[
|\overline A:C_{\overline A}(B)|=|\overline A|=|A:C_B|,
\]
and therefore $\lambda(B)=|A:C_B|$.
\end{proof}

\begin{corollary}[Coprime abelian split quotients: exact tuple count]\label{cor:coprime-exact-count}
Let $G=A\rtimes K$ with $A$ and $K$ finite abelian, and assume
\[
\gcd(|A|,|K|)=1.
\]
Then for every $r\ge 1$,
\[
|\Hom(\mathbb Z^r,G)|=\sum_{B\le K}|A:C_B|\,|C_B|^r\,\varphi_r(B)
=|A|\sum_{B\le K}|C_B|^{r-1}\,\varphi_r(B).
\]
Equivalently,
\[
|\Hom(\mathbb Z^r,G)|=\sum_{B\le K}\sum_{D\le B}|A:C_B|\,\mu_B(D)\,(|C_B|\,|D|)^r,
\]
and
\[
N_{\max}(G)=\sum_{\substack{B\le K\\ |C_B|\,|B|=m(G)}}|A:C_B|.
\]
\end{corollary}

\begin{proof}
For every subgroup $B\le K$ one has $\gcd(|A|,|B|)=1$, so Theorem~\ref{thm:coprime-lambda} gives
\[
\lambda(B)=|A:C_B|.
\]
Substituting this into Theorem~\ref{thm:abelian-exact-count} proves the exact tuple-count formulas, and substituting into Proposition~\ref{prop:abelian-coefficient} proves the formula for $N_{\max}(G)$.
\end{proof}

\begin{corollary}[Coprime fixed-point-free action]\label{cor:coprime-fixedfree}
Let $G=A\rtimes K$ with $A$ and $K$ finite abelian, assume $\gcd(|A|,|K|)=1$, and suppose
\[
\C_A(b)=1\qquad (1\neq b\in K).
\]
Then for every $r\ge 1$,
\[
|\Hom(\mathbb Z^r,G)|=|A|^r+|A|\,(|K|^r-1).
\]
Moreover,
\[
m(G)=\max\{|A|,|K|\}
\]
and
\[
N_{\max}(G)=
\begin{cases}
1,& |A|>|K|,\\[0.3em]
|A|,& |A|<|K|,\\[0.3em]
1+|A|,& |A|=|K|.
\end{cases}
\]
\end{corollary}

\begin{proof}
If $B=1$, then $C_B=A$. If $B>1$, choose $1\neq b\in B$; then
\[
C_B=\bigcap_{x\in B}\C_A(x)\le \C_A(b)=1,
\]
so $C_B=1$. Now Corollary~\ref{cor:coprime-exact-count} gives
\[
|\Hom(\mathbb Z^r,G)|=|A|^r+|A|\sum_{1<B\le K}\varphi_r(B).
\]
Every tuple in $K^r$ generates a unique subgroup of $K$, so
\[
|K|^r=\sum_{B\le K}\varphi_r(B)=1+\sum_{1<B\le K}\varphi_r(B),
\]
which proves the displayed formula for $|\Hom(\mathbb Z^r,G)|$. The formulas for $m(G)$ and $N_{\max}(G)$ follow from the explicit $C_B$ values in Proposition~\ref{prop:semidirect-largest} and Corollary~\ref{cor:coprime-exact-count}.
\end{proof}

\begin{proposition}[Abelian split quotients: exact leading coefficient]\label{prop:abelian-coefficient}
Let $G=A\rtimes K$ with $A$ and $K$ finite abelian, and retain the notation of Theorem~\ref{thm:abelian-exact-count}.
Then
\[
m(G)=\max_{B\le K}|C_B|\,|B|,
\qquad
N_{\max}(G)=\sum_{\substack{B\le K\\ |C_B|\,|B|=m(G)}}\lambda(B).
\]
Equivalently,
\[
|\Hom(\mathbb Z^r,G)|=
\left(\sum_{\substack{B\le K\\ |C_B|\,|B|=m(G)}}\lambda(B)\right)m(G)^r+O\!\bigl(b(G)^r\bigr).
\]
\end{proposition}

\begin{proof}
Because $K$ is abelian, Proposition~\ref{prop:semidirect-largest} gives
\[
m(G)=\max_{B\le K}|C_B|\,|B|.
\]
Now Theorem~\ref{thm:abelian-exact-count} gives
\[
|\Hom(\mathbb Z^r,G)|=\sum_{B\le K}\sum_{D\le B}\lambda(B)\mu_B(D)\,(|C_B|\,|D|)^r.
\]
The largest exponential base appearing in this finite sum is
\[
\max_{B\le K}|C_B|\,|B|=m(G).
\]
Moreover, if $D<B$, then $|D|<|B|$, so $|C_B|\,|D|<|C_B|\,|B|$.
Hence the coefficient of $m(G)^r$ comes only from the terms with $D=B$, and since $\mu_B(B)=1$ it is precisely
\[
\sum_{\substack{B\le K\\ |C_B|\,|B|=m(G)}}\lambda(B).
\]
This proves the formula for $N_{\max}(G)$, and the displayed asymptotic follows immediately.
\end{proof}

\begin{theorem}[Cyclic split quotients: exact tuple count]\label{thm:cyclic-exact-count}
Let $G=A\rtimes \langle t\rangle$ with $A$ finite abelian and $|t|=\omega$. For each divisor $n\mid \omega$ let
\[
B_n:=\langle t^{\omega/n}\rangle\le \langle t\rangle,
\qquad
C_n:=\C_A(B_n).
\]
Then for every $r\ge 1$,
\[
|\Hom(\mathbb Z^r,G)|=\sum_{n\mid \omega}|A:C_n|\,|C_n|^r\,J_r(n),
\]
where $J_r(n)$ is the $r$th Jordan totient, equivalently
\[
J_r(n)=\#\{(b_1,\dots,b_r)\in B_n^r:\ \langle b_1,\dots,b_r\rangle=B_n\}
=\sum_{d\mid n}\mu(d)\Bigl(\frac{n}{d}\Bigr)^r.
\]
Equivalently,
\[
|\Hom(\mathbb Z^r,G)|=
\sum_{n\mid \omega}\sum_{d\mid n}\mu(d)\,|A:C_n|\Bigl(\frac{|C_n|\,n}{d}\Bigr)^r.
\]
\end{theorem}

\begin{proof}
Write $A$ additively in this proof, and let $\pi:G\to \langle t\rangle$ be the natural projection. For each divisor $n\mid \omega$ set
\[
b_n:=t^{\omega/n},
\qquad
\mathcal T_n(r):=\{(x_1,\dots,x_r)\in \Hom(\mathbb Z^r,G):\ \langle \pi(x_1),\dots,\pi(x_r)\rangle=B_n\}.
\]
Then the sets $\mathcal T_n(r)$, for $n\mid \omega$, partition $\Hom(\mathbb Z^r,G)$.

For $a\in A$ define
\[
H_{n,a}:=C_n\,\langle (a,b_n)\rangle\le A\rtimes B_n.
\]
Because $B_n$ acts trivially on $C_n$, the subgroup $H_{n,a}$ is abelian. Also
\[
(a,b_n)^n=\Bigl(\sum_{j=0}^{n-1} b_n^{\,j}\!\cdot a,1\Bigr)\in C_n,
\]
because its $A$-component is fixed by $b_n$, hence by all of $B_n$. Consequently $H_{n,a}\cap A=C_n$. Moreover
\[
H_{n,a}=H_{n,a'}\iff a'-a\in C_n.
\]
Indeed, if $a'-a\in C_n$, then $(a',b_n)=(a'-a,1)(a,b_n)\in H_{n,a}$, so $H_{n,a'}=H_{n,a}$. Conversely, if $H_{n,a}=H_{n,a'}$, then any element of either subgroup projecting to $b_n$ has $A$-component in both cosets $a+C_n$ and $a'+C_n$, forcing $a'-a\in C_n$. Hence there are exactly $|A:C_n|$ distinct such subgroups.

We claim that every tuple in $\mathcal T_n(r)$ lies in a unique subgroup $H_{n,a}$. Let
\[
\mathbf x=(x_1,\dots,x_r)\in \mathcal T_n(r),
\qquad
H:=\langle x_1,\dots,x_r\rangle.
\]
Then $H$ is abelian and $\pi(H)=B_n$. Write $\pi(x_i)=b_n^{k_i}$ with $k_i\in \mathbb Z/n\mathbb Z$. Since the projected tuple generates $B_n$, there exist integers $u_1,\dots,u_r$ such that
\[
u_1k_1+\cdots+u_rk_r\equiv 1\pmod n.
\]
Because the $x_i$ commute, the element
\[
x:=x_1^{u_1}\cdots x_r^{u_r}\in H
\]
satisfies $\pi(x)=b_n$. Write $x=(a,b_n)$ for some $a\in A$. If $y\in H\cap A$, then $y$ commutes with $x$, so $y\in C_n$. Thus
\[
H=(H\cap A)\langle x\rangle\le C_n\langle x\rangle=H_{n,a},
\]
proving existence. For uniqueness, suppose that $\mathbf x$ lies in both $H_{n,a}$ and $H_{n,a'}$. Then the element $x$ above lies in both subgroups. But the elements of $H_{n,a}$ projecting to $b_n$ are exactly $(c+a,b_n)$ with $c\in C_n$, and similarly for $H_{n,a'}$. Hence the $A$-component of $x$ lies in both cosets $a+C_n$ and $a'+C_n$, so $a'-a\in C_n$ and therefore $H_{n,a}=H_{n,a'}$.

Fix now one of the $|A:C_n|$ distinct subgroups $H_{n,a}$. The projection map $H_{n,a}\to B_n$ is surjective with kernel $C_n$. Therefore, for each generating $r$-tuple in $B_n^r$, there are exactly $|C_n|^r$ lifts to an $r$-tuple in $H_{n,a}^r$. Hence the number of tuples in $H_{n,a}^r$ whose projection generates $B_n$ is
\[
|C_n|^r\,J_r(n),
\]
where $J_r(n)$ counts the generating $r$-tuples in the cyclic group $B_n$ of order $n$.
Summing over the $|A:C_n|$ distinct subgroups $H_{n,a}$ gives
\[
|\mathcal T_n(r)|=|A:C_n|\,|C_n|^r\,J_r(n).
\]
Finally, summing over all divisors $n\mid \omega$ proves the first displayed formula.
The second follows from the standard Möbius expansion of the Jordan totient,
\[
J_r(n)=\sum_{d\mid n}\mu(d)\Bigl(\frac{n}{d}\Bigr)^r.
\qedhere
\]
\end{proof}

\begin{proposition}[Cyclic split quotients: exact leading coefficient]\label{prop:cyclic-coefficient}
Let $G=A\rtimes \langle t\rangle$ with $A$ finite abelian and $|t|=\omega$. For each divisor $n\mid \omega$ let
\[
B_n:=\langle t^{\omega/n}\rangle\le \langle t\rangle,
\qquad
C_n:=\C_A(B_n).
\]
Then
\[
m(G)=\max_{n\mid \omega} |C_n|\,n,
\qquad
N_{\max}(G)=\sum_{\substack{n\mid \omega\\ |C_n|\,n=m(G)}} |A:C_n|.
\]
Equivalently,
\[
|\Hom(\mathbb Z^r,G)|=
\left(\sum_{\substack{n\mid \omega\\ |C_n|\,n=m(G)}} |A:C_n|\right)m(G)^r+O\!\bigl(b(G)^r\bigr).
\]
In particular, if a unique divisor $n_0$ of $\omega$ maximizes $|C_{n_0}|\,n_0$, then
\[
N_{\max}(G)=|A:C_{n_0}|.
\]
\end{proposition}

\begin{proof}
The formula for $m(G)$ is Proposition~\ref{prop:semidirect-largest} specialized to the cyclic quotient $\langle t\rangle$.
By Theorem~\ref{thm:cyclic-exact-count},
\[
|\Hom(\mathbb Z^r,G)|=
\sum_{n\mid \omega}\sum_{d\mid n}\mu(d)\,|A:C_n|\Bigl(\frac{|C_n|\,n}{d}\Bigr)^r.
\]
The largest exponential base appearing in this finite sum is
\[
\max_{n\mid \omega}|C_n|\,n=m(G),
\]
and it can occur only for the terms with $d=1$, because $d>1$ forces $(|C_n|n)/d<|C_n|n$.
Therefore the coefficient of $m(G)^r$ is precisely
\[
\sum_{\substack{n\mid \omega\\ |C_n|\,n=m(G)}} |A:C_n|,
\]
which proves the formula for $N_{\max}(G)$.
The displayed $\Hom$-asymptotic is then Theorem~\ref{thm:dominant-abelian}, and the final statement is immediate when the maximizing divisor $n_0$ is unique.
\end{proof}

\begin{corollary}[Uniform nontrivial fixed stratum]\label{cor:uniform-fixed-stratum}
Let $G=A\rtimes \langle t\rangle$ with $A$ finite abelian and $|t|=\omega>1$. Assume there exists a subgroup $F\le A$ such that
\[
\C_A(t^j)=F\qquad (1\le j\le \omega-1).
\]
Then
\[
m(G)=\max\{|A|,\,\omega|F|\}
\]
and
\[
N_{\max}(G)=
\begin{cases}
1,& |A|>\omega|F|,\\[0.3em]
|A:F|,& |A|<\omega|F|,\\[0.3em]
1+|A:F|,& |A|=\omega|F|.
\end{cases}
\]
\end{corollary}

\begin{proof}
For $n=1$ one has $C_1=A$. For every nontrivial divisor $n\mid \omega$ one has $C_n=F$, so Proposition~\ref{prop:cyclic-coefficient} gives
\[
m(G)=\max\left\{|A|,\ \max_{1<n\mid \omega} n|F|\right\}=\max\{|A|,\omega|F|\}.
\]
Among the nontrivial divisors, the unique subgroup of largest order is $\langle t\rangle$ itself, so only the strata $n=1$ and $n=\omega$ can contribute to $N_{\max}(G)$. The displayed three cases now follow directly from Proposition~\ref{prop:cyclic-coefficient}.
\end{proof}

\begin{example}[Two basic semidirect products]\label{ex:semidirect-basic}
The cyclic exact-count theorem gives closed formulas for two familiar split extensions.

\smallskip
\noindent\emph{(i) The symmetric group $S_3=C_3\rtimes C_2$.}
Let $A=C_3$ and $K=C_2=\langle t\rangle$, where $t$ acts on $A$ by inversion.
Then $C_1=A$, $C_2=1$, and $J_r(2)=2^r-1$, so Theorem~\ref{thm:cyclic-exact-count} gives
\[
|\Hom(\mathbb Z^r,S_3)|=|A:C_1|\,|C_1|^rJ_r(1)+|A:C_2|\,|C_2|^rJ_r(2)=3^r+3(2^r-1).
\]
In particular,
\[
P_r(S_3)=\frac{3^r+3(2^r-1)}{6^r}=2^{-r}+3^{1-r}-3\cdot 6^{-r}.
\]
Also Proposition~\ref{prop:semidirect-largest} gives $m(S_3)=3$, while Proposition~\ref{prop:cyclic-coefficient} gives
\[
N_{\max}(S_3)=1.
\]
Hence
\[
\lambda_{\mathrm{prob}}(S_3)=\frac{3}{6}=\frac12,
\qquad
\lambda_{\mathrm{orb}}(S_3)=3,
\qquad
h_{\mathrm{prob}}(S_3)=\log 2.
\]

\smallskip
\noindent\emph{(ii) The dihedral group $D_{2n}=C_n\rtimes C_2$ $(n\ge 3)$.}
Let $A=C_n$ and again let $K=C_2=\langle t\rangle$ act by inversion.
Then $C_1=A$ and
\[
C_2=\{a\in C_n:a=a^{-1}\}=\{a\in C_n:a^2=1\},
\]
so $|C_2|=\gcd(n,2)$. Therefore Theorem~\ref{thm:cyclic-exact-count} gives
\[
|\Hom(\mathbb Z^r,D_{2n})|=n^r+\frac{n}{\gcd(n,2)}\,\gcd(n,2)^r\,(2^r-1)
= n^r+n\,\gcd(n,2)^{r-1}(2^r-1).
\]
Equivalently,
\[
P_r(D_{2n})=2^{-r}+\frac{\gcd(n,2)^{r-1}(2^r-1)}{2^r n^{r-1}}.
\]
Moreover Proposition~\ref{prop:semidirect-largest} gives
\[
m(D_{2n})=\max\{n,\,2\gcd(n,2)\}=n,
\]
and Corollary~\ref{cor:uniform-fixed-stratum} yields
\[
N_{\max}(D_{2n})=
\begin{cases}
1,& n\neq 4,\\[0.3em]
3,& n=4.
\end{cases}
\]
Thus
\[
\lambda_{\mathrm{prob}}(D_{2n})=\frac{n}{2n}=\frac12,
\qquad
\lambda_{\mathrm{orb}}(D_{2n})=n,
\qquad
h_{\mathrm{prob}}(D_{2n})=\log 2.
\]
The second candidate $2\gcd(n,2)$ is the size of a maximum-order abelian subgroup containing a reflection; it equals $2$ for odd $n$, equals $4$ for even $n$, and ties the rotation subgroup only in $D_8$, where the three maximum-order abelian subgroups are $\langle r\rangle$, $\langle r^2,s\rangle$, and $\langle r^2,sr\rangle$.
\end{example}

\begin{remark}[Nonsplit abelian extensions]\label{rem:nonsplit-extension}
For a general extension $1\to A\to G\to K\to 1$ with $A$ abelian, the argument above still shows that every abelian subgroup $M\le G$ projecting onto $B\le K$ satisfies
\[
|M|\le |\C_A(B)|\,|B|.
\]
However, equality need not occur, because an abelian subgroup $B\le K$ need not lift to an abelian subgroup of $G$; the extension cocycle may obstruct such a lift. Thus split semidirect products are the natural exact next step after direct products.
\end{remark}

\begin{remark}[Entropy normalizations and isoclinism]\label{rem:entropy-isoclinism}
Because the full hierarchy $\{P_r(G)\}_{r\ge 2}$ is an isoclinism invariant, so are the normalized growth base $\lambda_{\mathrm{prob}}(G)$ and the entropy $h_{\mathrm{prob}}(G)$.
In particular, they cannot distinguish groups inside a fixed isoclinism family.
By contrast, the orbit-side normalization remembers abelian inflation.
If $A$ is abelian, then $G\times A$ is isoclinic to $G$ since
\[
(G\times A)/\Z(G\times A)\cong G/\Z(G),\qquad (G\times A)'=G',
\]
and one has
\[
\lambda_{\mathrm{prob}}(G\times A)=\lambda_{\mathrm{prob}}(G),\qquad
\lambda_{\mathrm{orb}}(G\times A)=|A|\,\lambda_{\mathrm{orb}}(G),
\]
equivalently,
\[
h_{\mathrm{prob}}(G\times A)=h_{\mathrm{prob}}(G),\qquad
h_{\mathrm{orb}}(G\times A)=h_{\mathrm{orb}}(G)+\log |A|.
\]
Thus the probability entropy measures the isoclinism class, whereas the orbit entropy differentiates only to the extent that it records the size of an abelian central inflation.
\end{remark}

\begin{corollary}[Minimal entropy stratum]\label{cor:min-entropy-stratum}
Let $G$ be a non-abelian finite group, and let $p$ be the smallest prime divisor of $|G|$.
Then
\[
\lambda_{\mathrm{prob}}(G)\le \frac{1}{p},
\qquad
h_{\mathrm{prob}}(G)\ge \log p.
\]
Equality holds if and only if $G$ has an abelian subgroup of index $p$.
In that case such a subgroup is automatically normal.
\end{corollary}

\begin{proof}
Any abelian subgroup of a non-abelian finite group is proper, so its index is at least $p$.
Hence $m(G)\le |G|/p$, which is equivalent to the displayed inequalities via Corollary~\ref{cor:entropy-bridge}.
Equality holds if and only if $m(G)=|G|/p$, i.e.\ if and only if $G$ has an abelian subgroup of index $p$.
If $A\le G$ has index $p$, then the action of $G$ on $G/A$ gives a transitive homomorphism $\varphi:G\to S_p$.
Its image $\varphi(G)$ has order divisible by $p$, while every prime divisor of $|\varphi(G)|$ also divides $|G|$ and is therefore at least $p$.
Since $|\varphi(G)|$ divides $p!$, this forces $|\varphi(G)|=p$.
A transitive permutation group of order $p$ acts regularly, so the stabilizer of a point is trivial.
Therefore the stabilizer $A$ equals $\ker(\varphi)$, and hence $A\trianglelefteq G$.
\end{proof}

\begin{remark}[Relation to earlier orbit-count growth results]\label{rem:kpp-growth}
Kaur--Prajapati--Prasad \cite{KaurPrajapatiPrasad2019} proved that the commuting-orbit generating series $B_G(t)$ from the introduction is rational and that $\kappa_r(G)$ has exponential growth factor $m(G)$.
Theorem~\ref{thm:dominant-abelian} sharpens this in two ways: it identifies the leading coefficient in the growth of $\kappa_r(G)$, and after dividing by $|G|^r$ it simultaneously yields the exact dominant term for $P_r(G)$. Thus the two root limits above are not independent new invariants but complementary normalizations of the same maximum-order abelian datum $m(G)$; the normalized base $m(G)/|G|$ is exactly the isoclinism-invariant normalization highlighted by Kaur--Prajapati--Prasad.
\end{remark}

\begin{corollary}[Rank-generating series]\label{cor:rank-series}
With $\mathcal P_G(z)$ and $B_G(t)$ as introduced in the introduction, one has
\[
\mathcal P_G(z)=\frac{B_G(z/|G|)-1}{z}.
\]
In particular $\mathcal P_G(z)$ is rational. Its radius of convergence is
\[
R_{\mathrm{prob}}(G)=\frac{|G|}{m(G)}.
\]
If $G$ is non-abelian, then $R_{\mathrm{prob}}(G)>1$, so both
\[
\Sigma_G:=\sum_{r\ge 2}P_r(G)
\qquad\text{and}\qquad
\Alt_G:=\sum_{r\ge 2}(-1)^rP_r(G)
\]
converge absolutely. If $G$ is abelian, then $\mathcal P_G(z)=1/(1-z)$, so neither series converges in the usual sense.
\end{corollary}

\begin{proof}
Using $P_r(G)=\kappa_{r-1}(G)/|G|^{r-1}$ for $r\ge 2$, we compute
\[
\mathcal P_G(z)
=\sum_{n\ge 0}\frac{\kappa_{n+1}(G)}{|G|^{n+1}}z^n
=\frac{1}{z}\sum_{n\ge 1}\kappa_n(G)\left(\frac{z}{|G|}\right)^n
=\frac{B_G(z/|G|)-1}{z}.
\]
Rationality now follows from \cite{KaurPrajapatiPrasad2019}. By the Cauchy--Hadamard formula and Corollary~\ref{cor:entropy-bridge},
\[
R_{\mathrm{prob}}(G)^{-1}
=\limsup_{n\to\infty}P_{n+2}(G)^{1/n}
=\lim_{r\to\infty}P_r(G)^{1/r}
=\frac{m(G)}{|G|},
\]
so $R_{\mathrm{prob}}(G)=|G|/m(G)$. If $G$ is non-abelian then $m(G)<|G|$, hence $R_{\mathrm{prob}}(G)>1$ and the evaluations at $z=1$ and $z=-1$ lie inside the disc of convergence. If $G$ is abelian, then $P_r(G)=1$ for all $r\ge 2$, hence $\mathcal P_G(z)=\sum_{n\ge 0}z^n=1/(1-z)$.
\end{proof}

\begin{theorem}[Finite Dirichlet spectrum and pole expansion]\label{thm:dirichlet-spectrum}
Assume that $G$ is non-abelian. Let $N_2(G)$ be the poset of abelian subgroups of $G$, and let $\mu_2$ be the Möbius function of $L_2(G):=N_2(G)\cup\{G\}$ in the sense of Torres-Giese \cite[Prop.~4.2]{TorresGiese2012}. For each integer $m\ge 2$ set
\[
c_m(G):=-\sum_{\substack{A\in N_2(G)\\ |G:A|=m}}\mu_2(A,G).
\]
Then only finitely many $c_m(G)$ are nonzero and, for every $r\ge 2$,
\begin{equation}\label{eq:dirichlet-spectrum}
P_r(G)=\sum_{m\ge 2}\frac{c_m(G)}{m^r}.
\end{equation}
Consequently,
\begin{equation}\label{eq:PG-partial-fractions}
\mathcal P_G(z)=\sum_{m\ge 2}\frac{c_m(G)}{m(m-z)}.
\end{equation}
If
\[
m_*:=\frac{|G|}{m(G)},
\]
then
\[
c_{m_*}(G)=N_{\max}(G).
\]
Equivalently, the first pole of $\mathcal P_G(z)$ is at $z=m_*=|G|/m(G)$, and the coefficient of $(m_*-z)^{-1}$ is
\[
\frac{N_{\max}(G)}{m_*}=\frac{N_{\max}(G)m(G)}{|G|}.
\]
\end{theorem}

\begin{proof}
In the notation of Torres-Giese, our $P_r(G)$ is $P_2(G,r)$; by \cite[Prop.~4.2]{TorresGiese2012},
\[
P_r(G)= -\sum_{A\in N_2(G)} \frac{\mu_2(A,G)}{|G:A|^r}.
\]
Grouping terms by the index $m=|G:A|$ gives \eqref{eq:dirichlet-spectrum}. Since $G$ has only finitely many subgroups, only finitely many $c_m(G)$ are nonzero.

Summing the geometric series for each fixed $m$ gives
\[
\sum_{r\ge 2}\frac{z^{r-2}}{m^r}
=\frac{1}{m^2}\sum_{n\ge 0}\Bigl(\frac{z}{m}\Bigr)^n
=\frac{1}{m(m-z)},
\]
which proves \eqref{eq:PG-partial-fractions}.

Now $m_*$ is the minimum index of an abelian subgroup of $G$. Hence every index $m\neq m_*$ that occurs in \eqref{eq:dirichlet-spectrum} satisfies $m>m_*$, so
\[
P_r(G)=c_{m_*}(G)\,m_*^{-r}+O(\theta^r)
\]
for some $\theta<m_*^{-1}$. On the other hand, Theorem~\ref{thm:dominant-abelian} gives
\[
P_r(G)=N_{\max}(G)\,m_*^{-r}+O\!\Bigl(\Bigl(\frac{b(G)}{|G|}\Bigr)^r\Bigr).
\]
Comparing the dominant coefficients shows that $c_{m_*}(G)=N_{\max}(G)$. The pole statement is immediate from \eqref{eq:PG-partial-fractions}.
\end{proof}

\begin{corollary}[Explicit Dirichlet coefficients for abelian split quotients]\label{cor:abelian-explicit-spectrum}
Assume that $G=A\rtimes K$ is non-abelian, with $A$ and $K$ finite abelian, and retain the notation of Theorem~\ref{thm:abelian-exact-count}. Then for every integer $m\ge 2$,
\[
c_m(G)=\sum_{\substack{B\le K,\ D\le B\\ m=|A:C_B|\,|K:D|}} \lambda(B)\mu_B(D).
\]
Equivalently, if
\[
m_{B,D}:=|A:C_B|\,|K:D|,
\]
then
\[
\mathcal P_G(z)=\sum_{B\le K}\sum_{D\le B}
\frac{\lambda(B)\mu_B(D)}{m_{B,D}(m_{B,D}-z)}.
\]
\end{corollary}

\begin{proof}
By Theorem~\ref{thm:abelian-exact-count},
\[
P_r(G)=\sum_{B\le K}\sum_{D\le B}\lambda(B)\mu_B(D)
\left(\frac{|C_B|\,|D|}{|A|\,|K|}\right)^r
=\sum_{B\le K}\sum_{D\le B}\lambda(B)\mu_B(D)\,m_{B,D}^{-r},
\]
where $m_{B,D}=|G|/(|C_B|\,|D|)=|A:C_B|\,|K:D|$.
If some $m_{B,D}=1$, then $|A:C_B|=1$ and $|K:D|=1$, so $C_B=A$ and $D=K$.
Hence $B=K$ and $K$ acts trivially on $A$, which would make $G=A\times K$ abelian, contrary to hypothesis.
Thus every $m_{B,D}\ge 2$.
Grouping equal indices $m_{B,D}$ gives the displayed formula for $c_m(G)$, and summing the geometric series as in the proof of Theorem~\ref{thm:dirichlet-spectrum} gives the partial-fraction expansion for $\mathcal P_G(z)$.
\end{proof}

\begin{corollary}[Explicit Dirichlet coefficients for cyclic split quotients]\label{cor:cyclic-explicit-spectrum}
Assume that $G=A\rtimes \langle t\rangle$ is non-abelian, with $|t|=\omega$, and retain the notation $C_n=\C_A(B_n)$ from Theorem~\ref{thm:cyclic-exact-count}. Then for every integer $m\ge 2$,
\[
c_m(G)=\sum_{\substack{n\mid \omega,\ d\mid n\\ m=|A:C_n|\,\omega d/n}}\mu(d)\,|A:C_n|.
\]
Equivalently, if
\[
m_{n,d}:=|A:C_n|\,\omega d/n,
\]
then
\[
\mathcal P_G(z)=\sum_{n\mid \omega}\sum_{d\mid n}
\frac{\mu(d)\,|A:C_n|}{m_{n,d}(m_{n,d}-z)}.
\]
\end{corollary}

\begin{proof}
By Theorem~\ref{thm:cyclic-exact-count},
\[
P_r(G)=\sum_{n\mid \omega}\sum_{d\mid n}\mu(d)\,|A:C_n|\left(\frac{|C_n|\,n/d}{|G|}\right)^r
=\sum_{n\mid \omega}\sum_{d\mid n}\mu(d)\,|A:C_n|\,m_{n,d}^{-r},
\]
where $m_{n,d}=|G|d/(|C_n|n)=|A:C_n|\,\omega d/n$.
If some $m_{n,d}=1$, then $|A:C_n|=1$ and $\omega d/n=1$, so necessarily $d=n=\omega$ and $C_\omega=A$.
Hence $t$ acts trivially on $A$, which makes $G=A\times C_\omega$ abelian, contrary to hypothesis.
Thus every $m_{n,d}\ge 2$, and grouping equal indices $m_{n,d}$ yields the displayed formula for $c_m(G)$.
Summing the geometric series as in the proof of Theorem~\ref{thm:dirichlet-spectrum} gives the partial-fraction expansion for $\mathcal P_G(z)$.
\end{proof}

\begin{corollary}[Linear recurrence, Hankel rank, and finite determination]\label{cor:recurrence-hankel}
Assume that $G$ is non-abelian, and let
\[
\{m\ge 2:c_m(G)\neq 0\}=\{m_1<\cdots<m_t\}.
\]
Write $\lambda_j:=m_j^{-1}$ for $1\le j\le t$, and let $\sigma_j$ be the $j$th elementary symmetric polynomial in $\lambda_1,\dots,\lambda_t$.
Then, for every $r\ge 2$,
\[
P_{r+t}(G)=\sigma_1P_{r+t-1}(G)-\sigma_2P_{r+t-2}(G)+\cdots+(-1)^{t-1}\sigma_tP_r(G).
\]
This recurrence has minimal order $t$. Equivalently, the infinite Hankel matrix
\[
\mathsf H_G:=\bigl(P_{i+j+2}(G)\bigr)_{i,j\ge 0}
\]
has rank $t$. In particular, once the spectral size $t$ is fixed, the initial block
\[
P_2(G),P_3(G),\dots,P_{2t+1}(G)
\]
determines the entire hierarchy $\{P_r(G)\}_{r\ge 2}$.
\end{corollary}

\begin{proof}
By Theorem~\ref{thm:dirichlet-spectrum},
\[
P_r(G)=\sum_{j=1}^t c_{m_j}(G)\,\lambda_j^r
\qquad (r\ge 2),
\]
with distinct scalars $\lambda_1,\dots,\lambda_t$ and nonzero coefficients $c_{m_j}(G)$.
Thus each summand satisfies the characteristic polynomial
\[
\Phi_G(x):=\prod_{j=1}^t(x-\lambda_j)=x^t-\sigma_1x^{t-1}+\sigma_2x^{t-2}-\cdots+(-1)^t\sigma_t,
\]
and summing over $j$ yields the stated recurrence.

Let
\[
H_t:=\bigl(P_{i+j+2}(G)\bigr)_{0\le i,j<t}.
\]
Then
\[
H_t=V\,D\,V^{\mathsf T},
\qquad
V=(\lambda_j^i)_{0\le i<t,\ 1\le j\le t},
\qquad
D=\operatorname{diag}\bigl(c_{m_1}(G)\lambda_1^2,\dots,c_{m_t}(G)\lambda_t^2\bigr).
\]
Hence
\[
\det(H_t)=\Bigl(\prod_{j=1}^t c_{m_j}(G)\lambda_j^2\Bigr)
\prod_{1\le i<j\le t}(\lambda_j-\lambda_i)^2\neq 0,
\]
so $\operatorname{rank}(\mathsf H_G)\ge t$.
On the other hand, the recurrence shows that every column of $\mathsf H_G$ after the first $t$ columns is a fixed linear combination of the preceding $t$ columns, so $\operatorname{rank}(\mathsf H_G)\le t$.
Thus $\operatorname{rank}(\mathsf H_G)=t$, and in particular no shorter linear recurrence can hold.

Finally, write the recurrence in the form
\[
P_{r+t}(G)=\beta_0P_r(G)+\beta_1P_{r+1}(G)+\cdots+\beta_{t-1}P_{r+t-1}(G)
\qquad (r\ge 2).
\]
For $r=2,3,\dots,t+1$ this gives the linear system
\[
H_t\begin{pmatrix}\beta_0\\ \beta_1\\ \vdots\\ \beta_{t-1}\end{pmatrix}
=
\begin{pmatrix}P_{t+2}(G)\\ P_{t+3}(G)\\ \vdots\\ P_{2t+1}(G)\end{pmatrix}.
\]
Since $H_t$ is invertible, the coefficients $\beta_0,\dots,\beta_{t-1}$ are uniquely recovered from $P_2(G),\dots,P_{2t+1}(G)$.
Those coefficients together with $P_2(G),\dots,P_{t+1}(G)$ then determine all later values recursively.
\end{proof}

\begin{corollary}[Inverse finite-spectrum rigidity]\label{cor:inverse-spectrum}
Let $G$ and $H$ be non-abelian finite groups. If
\[
P_r(G)=P_r(H)
\]
for infinitely many integers $r\ge 2$ (in particular, for all sufficiently large $r$), then
\[
c_m(G)=c_m(H)\qquad\text{for every }m\ge 2.
\]
Equivalently,
\[
\mathcal P_G(z)=\mathcal P_H(z).
\]
Hence the finite abelian-index spectra coincide, and in particular
\[
\frac{m(G)}{|G|}=\frac{m(H)}{|H|},
\qquad
N_{\max}(G)=N_{\max}(H).
\]
If moreover $|G|=|H|$, then $m(G)=m(H)$ and
\[
\kappa_r(G)=\kappa_r(H)\qquad\text{for all }r\ge 1.
\]
\end{corollary}

\begin{proof}
Set
\[
a_m:=c_m(G)-c_m(H),
\]
so that only finitely many $a_m$ are nonzero and, for every $r\ge 2$ in the given infinite set,
\[
0=P_r(G)-P_r(H)=\sum_{m\ge 2}\frac{a_m}{m^r}.
\]
Assume that some $a_m$ is nonzero, and let $m_0$ be the smallest index with $a_{m_0}\neq 0$.
Multiplying by $m_0^r$ gives
\[
0=a_{m_0}+\sum_{m>m_0} a_m\Bigl(\frac{m_0}{m}\Bigr)^r.
\]
Along the infinite sequence of admissible values of $r$, the second term tends to $0$ because $0<m_0/m<1$ for every $m>m_0$.
Passing to the limit yields $a_{m_0}=0$, a contradiction.
Hence all $a_m=0$, proving the coefficient identity and therefore also $\mathcal P_G(z)=\mathcal P_H(z)$ via \eqref{eq:PG-partial-fractions}.

The equality of first poles gives $|G|/m(G)=|H|/m(H)$, and the equality of the corresponding pole coefficients gives $N_{\max}(G)=N_{\max}(H)$ by Theorem~\ref{thm:dirichlet-spectrum}.
If $|G|=|H|$, then the equality of ratios yields $m(G)=m(H)$, and since $P_r(G)=P_r(H)$ for all $r\ge 2$, Theorem~\ref{thm:burnside} gives
\[
\kappa_r(G)=|G|^rP_{r+1}(G)=|H|^rP_{r+1}(H)=\kappa_r(H)
\]
for every $r\ge 1$.
\end{proof}

\begin{corollary}[Special values at \texorpdfstring{$z=\pm 1$}{z=+/-1}]\label{cor:special-values}
If $G$ is non-abelian, then
\[
\Sigma_G=\sum_{m\ge 2}\frac{c_m(G)}{m(m-1)},
\qquad
\Alt_G=\sum_{m\ge 2}\frac{c_m(G)}{m(m+1)}.
\]
\end{corollary}

\begin{proof}
Evaluate \eqref{eq:PG-partial-fractions} at $z=1$ and $z=-1$.
\end{proof}

\begin{remark}[Dirichlet value versus ordinary generating value]\label{rem:dirichlet-vs-ordinary}
The topological specialization $P_2(G,-1)=\chi(E(2,G))$ of Torres-Giese \cite[Thm.~1.1]{TorresGiese2012} is different from the alternating sum $\Alt_G=\mathcal P_G(-1)$. In the notation above,
\[
P_2(G,-1)=\sum_{m\ge 2} c_m(G)\,m,
\qquad
\Alt_G=\sum_{m\ge 2}\frac{c_m(G)}{m(m+1)}.
\]
Thus the same finite abelian-index spectrum $\{(m,c_m(G))\}$ simultaneously controls the Dirichlet value, the ordinary generating-series value, and the asymptotic pole data.
\end{remark}

\begin{remark}[Root geometry of the truncation polynomials]\label{rem:EK-truncations}
For $N\ge 2$ let
\[
\mathcal P_{G,N}(z):=\sum_{r=2}^{N}P_r(G)z^{r-2}=\sum_{j=0}^{N-2} a_j z^j,\qquad a_j:=P_{j+2}(G).
\]
By Proposition~\ref{prop:monotone}, the coefficients satisfy
\[
a_0\ge a_1\ge \cdots \ge a_{N-2}>0.
\]
Hence the reversed polynomial
\[
z^{N-2}\mathcal P_{G,N}(1/z)=\sum_{j=0}^{N-2} a_{N-2-j}z^j
\]
has nondecreasing positive coefficients. By the Enestr\"om--Kakeya theorem, and equally by the Joyal--Labelle--Rahman extension for real monotone coefficients \cite{GovilRahman1968,JoyalLabelleRahman1967}, all zeros of the reversed polynomial lie in $|z|\le 1$. Equivalently, every zero $\zeta$ of $\mathcal P_{G,N}(z)$ satisfies
\[
|\zeta|\ge 1.
\]
Thus the partial rank-growth polynomials are zero-free in the open unit disk. In the abelian case one has $\mathcal P_{G,N}(z)=1+z+\cdots+z^{N-2}$, so the zeros lie exactly on the unit circle.
\end{remark}

\medskip

\section{Conclusion}

The main point of this paper is that higher commutativity in finite groups has an exact dominant term governed by the geometry of the maximum-order abelian subgroups. The asymptotic formula, the finite abelian-index spectrum for $\mathcal P_G(z)$, and the first-pole coefficient theorem show that the hierarchy has a rigid exponential base and a finite spectral support.

The same spectrum also forces a linear recurrence of exact order equal to the number of nonzero spectral coefficients, equivalently a finite-rank Hankel matrix for the hierarchy. Thus, once the spectral size is fixed, a finite initial block determines the entire tower $\{P_r(G)\}_{r\ge 2}$.

The split-extension formula
\[
m(G)=\max_{B\le K\ \mathrm{abelian}} |\C_A(B)|\,|B|
\]
shows that the dominant asymptotic base is already visible in semidirect products, not only in direct-product or class-$2$ examples. We now go further for abelian split quotients: if $G=A\rtimes K$ with $K$ abelian, then the full hierarchy admits the exact subgroup-lattice expansion
\[
|\Hom(\mathbb Z^r,G)|=\sum_{B\le K}\lambda(B)\,|\C_A(B)|^r\,\varphi_r(B),
\]
and therefore the entire finite Dirichlet spectrum is explicit up to the coefficients $\lambda(B)$; in the cyclic and coprime cases those coefficients are completely explicit. The leading coefficient is the top stratum of this exact expansion,
\[
N_{\max}(G)=\sum_{\substack{B\le K\\ |\C_A(B)|\,|B|=m(G)}}\lambda(B).
\]
In the coprime-action regime $(|A|,|K|)=1$, Schur--Zassenhaus forces a uniform collapse
\[
\lambda(B)=|A:\C_A(B)|\qquad (B\le K),
\]
so both the exact tuple count and the leading coefficient are determined entirely by fixed-point data. Thus the remaining ambiguity in the general abelian-split formula is concentrated in the genuinely non-coprime case. The low-rank loop-groupoid, Drinfeld-double, and quantum-triple formulas are now entirely deferred to the companion paper \cite{LevitShwartzPartII}, so the center of gravity here is purely the finite-group asymptotic/spectral package. A natural next step is to understand the non-coprime lift multiplicities $\lambda(B)$ in the same explicit spirit.

\end{document}